\documentclass[12pt]{amsart}
\usepackage{amsmath,amssymb,amsbsy,amsfonts,amsthm,latexsym,
                        amsopn,amstext,amsxtra,euscript,amscd,mathrsfs,color,bm}
                        
\usepackage{float} 
\usepackage[english]{babel}
\usepackage{url}
\usepackage[colorlinks,linkcolor=blue,anchorcolor=blue,citecolor=blue,backref=page]{hyperref}
\usepackage[norefs,nocites]{refcheck}

 \usepackage[np]{numprint}

\npdecimalsign{\ensuremath{.}}

\usepackage{bibentry}

\usepackage[english]{babel}
\usepackage{mathtools}
\usepackage{todonotes}
\usepackage{url}
\usepackage[colorlinks,linkcolor=blue,anchorcolor=blue,citecolor=blue,backref=page]{hyperref}

\begin{document}

\newcommand{\commA}[2][]{\todo[#1,color=yellow]{A: #2}}
\newcommand{\commI}[2][]{\todo[#1,color=green!60]{I: #2}}

\theoremstyle{plain}
\newtheorem{Th}{Theorem}[section]
\newtheorem{Lemma}[Th]{Lemma}
\newtheorem{Cor}[Th]{Corollary}
\newtheorem{Prop}[Th]{Proposition}

 \theoremstyle{definition}
\newtheorem{Def}[Th]{Definition}
\newtheorem{Conj}[Th]{Conjecture}
\newtheorem{Rem}[Th]{Remark}
\newtheorem{?}[Th]{Problem}
\newtheorem{Ex}[Th]{Example}

\newcommand{\im}{\operatorname{im}}
\newcommand{\Hom}{{\rm{Hom}}}
\newcommand{\diam}{{\rm{diam}}}
\newcommand{\ovl}{\overline}

\title[Semigorup orbits and multiplicative orders]{ On semigroup orbits of polynomials  and multiplicative orders}

\author[Jorge Mello]{Jorge Mello}

\address{UNSW} 

\email{j.mello@unsw.edu.au}

 \subjclass[2010]{Primary}


\begin{abstract} We show, under some natural restrictions, that some semigroup orbits of polynomials
cannot contain too many elements of small multiplicative order modulo a large prime
$p$, extending previous work of Shparlinski (2017).
\end{abstract}

\maketitle
\section{introduction}
Let $K$ be a field and $\overline{K}$ its algebraic closure. Let $\mathcal{F}= \{ \phi_1,..., \phi_k \} \subset K[X]$ be a set of polynomials of degree at least $2$, let $x \in K$, and let 
\begin{center}$\mathcal{O}_{\mathcal{F}}(x)= \{ \phi_{i_n} \circ ... \circ \phi_{i_1}(x) | n \in \mathbb{N} , i_j =1,...,k. \}$\end{center} denote the forward orbit of $P$ under $\mathcal{F}$.

For a prime $p$ and an integer $s \geq 1$, let $\mathbb{F}_{p^s}$ denote the finite field of $p^s$ elements. Then, for $w \in  \mathbb{F}_{p^s}$ and $\phi_1,..., \phi_k$ defined over $\mathbb{F}_{p^s}$
 \begin{center}
  $T(w):= \# \mathcal{O}_{\mathcal{F}}(w)\leq p^s$.
 \end{center}
 For each $u \in \overline{\mathbb{F}_p}^*$, we define the multiplicative order $\tau(u)$ as the smallest integer $l \geq 1$ such that $u^l=1$.

When $k=1$ and $\epsilon>0$ is arbitrary, Shparlinski \cite{S} proved the bound
\begin{center}
 $\# \{ n \leq N-1 : \tau(f^{(n)}(x)) \leq t \} \leq O(\max \{N^{1/2}, N/ \log \log p  \})$,
\end{center} for all $ x \in \overline{\mathbb{F}}_p$, $p$ prime and $t  \leq (\log p)^{1/2 - \epsilon}$, where $f$ is not linear conjugate to monomials neither to Chebyshev polynomials.

In this paper we seek to generalise results of this sort when the dynamical systems are generated as  semigroups under composition by several maps $\phi_i$  not linear conjugate to monomials neither to Chebyshev polynomials.
 initially. Precisely, putting $\mathcal{F}_n= \{\phi_{i_n}\circ ... \circ \phi_{i_1}| 1 \leq i_j \leq k  \}$ for the $n$-level set, $\epsilon >0$, and  $h \geq 3l$ integers  satisfying that $B(k,h)\geq ck^h$ ( $c$ is a positive constant) and
\begin{center}
  $\# \{ v \in  \overline{\mathbb{F}}_p| \tau (v) \leq t, v = f(u), f \in \mathcal{F}_n, n \leq N   \} \geq B(k,h)$
 \end{center} for each $N$,
 we prove among other results that
\begin{align*}
\# \{ v \in   \overline{\mathbb{F}}_p |\tau (v) \leq t, v = f(u), f \in &\mathcal{F}_n, n \leq N   \} \\ &\ll_{l, \mathcal{F}} \max \left\{ \dfrac{B(k,h)^{l+1}}{ h} ,  \dfrac{ B(k,h)^{l+1}}{ \log \log p} \right\}, 
\end{align*} for all $ x \in \overline{\mathbb{F}}_p$, $p$ prime and $t  \leq (\log p)^{1/2 - \epsilon}$.
Namely, if the number of orbit points of iteration order at most $N$ and multiplicative order at most $t$ is bigger than a multiple of the size of the complete $k$-tree of depth $h-1$, then such pursued number is upper bounded in terms of $B(k,h)$ and the charateristic of the field as stated above.

For doing this, we use recent results of Ostafe and Young \cite{OY} about the finiteness of cyclotomic algebraic points that are preperiodic for $\mathcal{F}$ and that fall on the set of roots of unity, and of graph theory due to M\'{e}rai and Shparlinski \cite{MS}.

Sections \ref{sec2}, \ref{sec3} and \ref{sec5} are devoted to preliminar notation and results. Section \ref{sec4} contains results for points in orbits generated by sequences of polynomials from the initial set, and Section \ref{sec6} contains the result for the full semigroup orbit.

\section{Preliminar notations}\label{sec2} Let $K$ be a field  and $\overline{K}$ its algebraic closure. For $K$ a number field and $x \in K $, the naive logarithmic height $h(x)$ is given by 
\begin{center}$ \sum_{v \in M_K} \dfrac{[K_v: \mathbb{Q}_v]}{[K:\mathbb{Q}]} \log(\max \{1, |x|_v\}$, \end{center}
where $M_K$ is the set of places of $K$, $M_K^\infty$ is the set of archimedean (infinite) places of $K$, $M_K^0$ is the set of nonarchimedean (finite) places of $K$, and for each $v \in M_K$, $|.|_v$ denotes the corresponding absolute value on $K$ whose restriction to $\mathbb{Q}$ gives the usual $v$-adic absolute value on $\mathbb{Q}$.
Also, we write $K_v$ for the completion of $K$ with respect to $|.|$, and we let $\mathbb{C}_v$ denote the completion of an algebraic closure of $K_v$. To simplify notation, we let $d_v=[K_v:\mathbb{Q}_v]/[K:\mathbb{Q}]$.

For an arbitrary field $K$, let $\mathcal{F}= \{ \phi_1,..., \phi_k \} \subset K[X]$ be a dynamical system of polynomials of degree at least $2$. For $x \in K$,  let \begin{center} $\mathcal{O}_{\mathcal{F}}(x)= \{ \phi_{i_n} \circ ... \circ \phi_{i_1}(x) | n \in \mathbb{N} , i_j =1,...,k. \}$ \end{center} denote the forward orbit of $P$ under $\mathcal{F}$.

We set $J=\{ 1,...,k \}, W= \prod_{i=1}^\infty J$, and $\Phi_w:=(\phi_{w_j})_{j=1}^\infty$ to be a sequence of polynomials from $\mathcal{F}$ for $w= (w_j)_{j=1}^\infty \in W$. 
In this situation we let \begin{center}$\Phi_w^{(n)}=\phi_{w_n} \circ ... \circ \phi_{w_1}$ with $\Phi_w^{(0)}=$Id, 
and also $\mathcal{F}_n :=\{ \Phi_w^{(n)} | w \in W \}$.\end{center}

Precisely, we consider polynomials sequences $\Phi$ $= (\phi_{i_j})_{j=1}^\infty \in \prod_{i=1}^\infty \mathcal{F}$ and $x \in \overline{K}$,
denoting \begin{center}$\Phi^{(n)}(x):=\phi_{i_n}(\phi_{i_{n-1}}(...(\phi_{i_1}(x)))$.\end{center} The set \begin{align*}\{ x, \Phi^{(1)}(x),  \Phi^{(2)}(x),  \Phi^{(3)}(x),... \} 
 =\{ x, \phi_{i_1}(x), \phi_{i_2}(\phi_{i_1}(x)), \phi_{i_3}(\phi_{i_2}(\phi_{i_1}(x)),... \}\end{align*} is called the forward orbit of $x$ under $\Phi$, denoted by
$\mathcal{O}_{\Phi} (x)$. 

The point $x$ is said to be $\Phi$-preperiodic if $\mathcal{O}_{\Phi} (x)$ is finite.

 For a $x \in K$, the $\mathcal{F}$-orbit of $x$ is defined as 
\begin{align*}
 \mathcal{O}_{\mathcal{F}}(x)=\{ \phi(x) | \phi \in \bigcup_{n \geq 1} \mathcal{F}_n \}= \{ \Phi_w^{(n)}(x) | n \geq 0, w \in W \} = \bigcup_{w \in W} \mathcal{O}_{\Phi_w} (x).
 \end{align*}
 The point $x$ is called preperiodic for $\mathcal{F}$ if $\mathcal{O}_{\mathcal{F}}(x)$ is finite.
 
 We let $S$ be the \textit{shift map} which sends $\Psi$ $=(\psi_i)_{i=1}^\infty$ to \begin{center}$S(\Psi) =(\psi_{i+1})_{i=1}^\infty$.\end{center}
 
 For $\mathcal{S} \subset K$ and an integer $N \geq 1$, we use $T_{x,\Phi}(N, \mathcal{S})$ to denote the number of $n \leq N$ with $\Phi^{(n)}(w) \in \mathcal{S}$, namely, 
 \begin{center}
  $T_{x,\Phi}(N, \mathcal{S}) = \# \{ n \leq N | \Phi^{(n)}(x) \in \mathcal{S} \}$.
 \end{center}
 For $f = \sum_{i=0}^d a_i X^i \in \overline{\mathbb{Q}}[X]$ and $K$ a number field containing all the coefficients of $f$, denote the weil height of $f$ by 
 \begin{center}
  $h(f) = \sum_{v \in M_K} d_v \log(\max_i |a_i|_v)$,
 \end{center} and for the system of polynomials $\mathcal{F}= \{ \phi_1,..., \phi_k \}$, denote $h(\mathcal{F})=\max_i h({\phi_i})$. 
 
We revisit the following bound calculated in other works, for example, \cite[Proposition 3.3]{M}.

 \begin{Prop}\label{prop2.1}
  Let $\mathcal{F}= \{ \phi_1,..., \phi_k \}$ be a finite set of polynomials over  a number field $K$ with $\deg \phi_i= d_i \geq 2$, and $d:= \max_i d_i$. Then for all $n \geq 1$ and $\phi \in \mathcal{F}_n$, we have
  \begin{center}
   $h(\phi) \leq (\dfrac{d^n-1}{d-1})h(\mathcal{F}) + d^2(\dfrac{d^{n-1}-1}{d-1})\log 8= O(d^n(h(\mathcal{F})+1))$.
  \end{center}
 \end{Prop}

\section{Some preliminar results}\label{sec3}

We also will make use of the combinatorial statement below, which has been used and proved in a number of works.
\begin{Lemma}\label{lem3.1}\cite[Lemma 4.8]{OS}
 Let $2 \leq T < N/2$. For any sequence 
 \begin{center}
  $0 \leq n_1 < ... < n_T \leq N$,
 \end{center}there exists $r \leq 2N/T$ such that $n_{i+1}- n_i=r$ for at least $T(T-1)/4N$ values of $i \in \{ 1,...,T-1 \}$.
\end{Lemma}
The following result for more general fields is a direct application of the previous lemma.

\begin{Prop}\label{prop3.2}
 Let $K$ be an arbitrary field, $x \in K$ and let $\mathcal{S} \subset K$ be an arbitrary subset of $K$. Suppose there exist a  real number $0< \tau < 1/2$, and also $\Phi$ a sequence of polynomials contained in $\mathcal{F}= \{ \phi_1,.., \phi_k \} \subset K[X]$ such that
 \begin{center}
  $T_{x,\Phi}(N, \mathcal{S})= \tau N \geq 2$.
 \end{center} Then there exists an integer $t \leq 2 \tau^{-1}$ such that
 \begin{center}
  $\# \{ (u,v) \in \mathcal{S}^2 |  (S^n\Phi)^{t} (u)=v \text{ for some } n  \} \geq \dfrac{\tau^2N}{8}$.
 \end{center}

\end{Prop}

\begin{proof}
 Letting $T:= T_{x,\Phi}(N, \mathcal{S})$, we consider all the values $1\leq n_1 <...< n_T \leq N$ such that $\Phi^{(n_i)}(x) \in \mathcal{S}, i=1,...,T-1$.
 
 From the previous lemma, there exists $t \leq 2 \tau^{-1}$ such that the number of $i=1,..., T-1$ with $n_{i+1}- n_i=t$ is at least
 \begin{center}
  $\dfrac{T(T-1)}{4N} = \dfrac{T^2}{4}\left( 1 - \dfrac{1}{T} \right) = \dfrac{\tau^2 N}{4} \left( 1 - \dfrac{1}{T} \right) \geq \dfrac{\tau^2N}{8}$.
 \end{center} Moreover, if $\mathcal{J}:= \{ 1 \leq j \leq T-1 | n_{j+1} - n_j = t \}$,  then
 \begin{center}
  $\Phi^{(n_j)}(x) \in \mathcal{S}$ and $\Phi^{(n_{j+1})}(x)= (S^{n_j}\Phi)^{t}(\Phi^{(n_j)}(x)) \in \mathcal{S}$ for each $ j \in \mathcal{J}$,
 \end{center} and hence 
 \begin{center}
  $\# \{ (u,v) \in \mathcal{S}^2 |  (S^n\Phi)^{t} (u)=v \text{ for some } n  \}  \geq \dfrac{\tau^2N}{8}$.
 \end{center}
\end{proof}

\section{Multiplicative orders in finite fields}\label{sec4}

 In this section we consider $\mathcal{F} = \{\phi_1,...,\phi_k \} \subset \mathbb{Z}[X]$. For a prime $p$ and an integer $s \geq 1$, let $\mathbb{F}_{p^s}$ denote the finite field of $p^s$ elements. Then, for $w \in  \mathbb{F}_{p^s}$ and $\phi_1,..., \phi_k$ defined over $\mathbb{F}_{p^s}$
 \begin{center}
  $T(w):= \# \mathcal{O}_{\mathcal{F}}(w)\leq p^s$.
 \end{center}
 For each $u \in \overline{\mathbb{F}_p}^*$, we define the multiplicative order $\tau(u)$ as the smallest integer $l \geq 1$ such that $u^l=1$. 
 
 For $\Psi$ a sequence of terms in $\mathcal{F}$, we also use the notation
 \begin{center}
  $M_{w,\Psi}(t,N)= \# \{ n \leq N-1 : \tau(\Psi^{(n)}(w)) \leq t) \}$.
  \end{center}
  We use $\mathbb{U}$ to denote the set of all roots of unity in $\mathbb{C}$ and $\Phi_s$ to denote the $s$th
 cyclotomic polynomial. The resultant of two polynomials $F,G \in \mathbb{Z}[X]$ is denoted by Res$(F,G)$. The lemma below is a well known fact.
 \begin{Lemma}\label{lem4.1}\cite[Lemma 2.6]{S}
  For any integers $r,s \geq 1$ and $F \in \mathbb{Z}[X]$, we have 
  \begin{center}
   Res$(\Phi_r,\Phi_s(F))= \exp(O(rs(h(F) + \deg F)))$.
  \end{center}
\end{Lemma}

 We now formulate special cases of results due to Ostafe and Young \cite{OY}.
 
 \begin{Lemma}\label{lem4.2} 
  Let $\mathcal{F}=\{ \phi_1,..., \phi_k \} \in \mathbb{Z}[X]$ be a set of polynomials of respective degrees $d_i \geq 2$ that are not linear conjugate to monomials neither to Chebyshev polynomials (non-special). Then $\phi_i (\mathbb{Q}(\mathbb{U}))$ is finite for each $i$ and so is the set of $u$ is $\mathbb{Q}(\mathbb{U})$ such that $ \cup_{n \geq 0}\mathcal{F}_n(u) \cap \mathbb{U} \neq 0$. 
 \end{Lemma}
 \begin{proof}
  Use Corollary 2.2 of \cite{OS} and Theorem 1.4 of \cite{OY}.
 \end{proof}

 \begin{Lemma}\label{lem4.3}\cite[Theorem 1.7]{OY}
  Under the conditions of the previous lemma, we have
  \begin{center}
   $\{ \alpha \in \mathbb{Q}(\mathbb{U}) | \phi_{i_s}\circ ... \circ \phi_{i_1}(\alpha) \in \mathcal{F}_l (\phi_{i_s}\circ ... \circ \phi_{i_1}(\alpha)), s \geq 0, l \geq 1  \} $ is finite.
  \end{center}

 \end{Lemma}
 Now we want to bound the number of elements of bounded order in an orbit.
 
 \begin{Th}\label{th4.4}
  Let $\mathcal{F}=\{ \phi_1,..., \phi_k \} \in \mathbb{Z}[X]$ be a set of polynomials of respective degrees $d_i \geq 2$ that are not linear conjugate to monomials neither to Chebyshev polynomials (non-special), and $d = \max_i \deg_i$.
  Then, for any fixed $\epsilon >0$,
  
  (i) For any prime $p$ and $t\leq (\log p) ^{1/2 - \epsilon}$, for all initial values $w \in \overline{\mathbb{F}_p}$, we have 
  \begin{center}
   $\displaystyle\sup_{\Psi \text{ sequence in } \mathcal{F}} M_{w, \Psi}(t,N) = O(\max \{N^{1/2} , N/\log \log p \})$.
  \end{center}
  (ii) Let $\Psi$ be a sequence of terms in $\mathcal{F}$. Then for any sufficiently large $P \geq 1$ and $t \leq P^{1/2 - \epsilon}$, for all but $o(P/\log P)$ primes $p \leq P$, for all initial values $w \in \overline{\mathbb{F}}_p$, we have 
  \begin{center}
   $M_{w, \Psi}(t,N) =O(\max \{ N^{1/2}, N/ \log p  \})$
  \end{center}

 \end{Th}
 \begin{proof}
  For $w \in \overline{\mathbb{F}}_p$ and $\Psi$ sequence of terms  in $\mathcal{F}$, denote
  \begin{center}
   $M_{w, \Psi}(t,N):= \rho_{\Psi} N$.
   \end{center} Then there are at least $\rho_\Psi N$ values of $n <N$ with $\Phi_l(\Psi^{(n)}(w))=0$ for some positive integer $l \leq t$.
By Proposition~\ref{prop3.2}, there is some positive integer $m_\Psi \leq 2 \rho_\Psi^{-1}$ such that for at least $\rho_\Psi^2N/8$ values $u \in \overline{\mathbb{F}}^*_p$, we have 
\begin{center}
 $\Phi_s(u)= \Phi_l(\gamma(u))=0$
\end{center} for some pair $(k,l) \in [1,t]^2, \gamma=(S^n\Psi)^{(m_\Psi)} \in \mathcal{F}_{m_\Psi}$. Denote by $R_{s,l,{m_\Psi},\gamma}$ the resultant of the polynomials $\Phi_s(X)$ and $ \Phi_l(\gamma(X))$.
By Lemma~\ref{lem4.2}, there are only finitely many values of ${m_\Psi}$ for which $R_{s,l,{m_\Psi},\gamma}=0$ is possible for some $s, l$ and $\gamma$. Then there are at most $c_1$ ( does not depend on $\Psi$, but only on $\mathcal{F}$) values of $u \in  \overline{\mathbb{F}}_p$
solutions of $R_{s,l,{m_\Psi},\gamma}=0$. If $ \rho_\Psi^2N/8> c_1$, there exists a $(s,l,{m_\Psi}, \gamma)$ such that $p | R_{s,l,{m_\Psi},\gamma} \neq 0$.

If this is the case , then $\rho_\Psi > \sqrt{8c_1/N}$. Using Lemma~\ref{lem4.1} and Proposition~\ref{prop2.1} with $d = \max_i d_i$, we have
\begin{center}
 $\log |R_{s,l,{m_\Psi},\gamma}|=O(sld^{m_\Psi})=O(t^2d^{2{\rho_\Psi}^{-1}})$,
\end{center} and this does not depend on $p$ nor in which $\Psi$ was taken initially.

Thus $\log p = O(t^2d^{2{\rho_\Psi}^{-1}})$, and as $t \leq (\log p)^{1/2 - \epsilon}$, we derive $(\log p)^{2 \epsilon} \leq t^{-2}\log p = O(d^{2 {\rho_\Psi}^{-1}})$, and therefore
 $\rho_\Psi \leq c_2 (\log \log p)^{-1}$. Taking
 \begin{center}
  $\rho_\Psi = \max \{ 3 \sqrt{c_1/N}, 2 c_2 (\log \log p)^{-1} \}$
 \end{center} (for a constant $c_2$ that depends only on $\mathcal{F}$) makes the calculations above induce a contradiction, from where we proved part (i).

 For (ii), denoting by $\Omega(r)$ the number of distinct prime divisors of an integer $r \neq 0$, we see that if $R_{s,l,{m_\Psi},\gamma} \neq 0$, then
\begin{center}
 $\Omega(R_{s,l,{m_\Psi},\gamma}) \leq 2 \log |R_{s,l,{m_\Psi},\gamma}| \leq O(t^2d^{2{\rho_\Psi}^{-1}}) \leq O( P^{1- 2 \epsilon} d^{2{\rho_\Psi}^{-1}})$,
\end{center} and this does not depend on $p$ nor in which $\Psi$ was taken initially.

We note that
$\rho_\Psi \geq \dfrac{2\log d}{\epsilon \log P}$ implies that $\Omega(R_{s,l,{m_\Psi},\gamma}) = o(P/ \log P)$. Therefore, $\displaystyle\max_{\Phi} \rho_\Phi \leq O(1/\log p)$ for all but $o(P/\log P)$ primes in this case.

 \end{proof}

\begin{Cor}
 Let $\mathcal{F}=\{ \phi_1,..., \phi_k \} \in \mathbb{Z}[X]$ be a set of polynomials of respective degrees $d_i \geq 2$ that are not linear conjugate to monomials neither to Chebyshev polynomials (non-special) and $d = \max_i \deg f_i$.
  Then, for any fixed $\epsilon >0$,
  
   For any prime $p$ and $t\leq (\log p) ^{1/2 - \epsilon}$, for all initial values $w \in \overline{\mathbb{F}_p}$, we have 
\begin{align*}
 \# \{ u \in \overline{\mathbb{F}}_p| u= f(w), \tau(u) \leq t ,f \in &\mathcal{F}_n, n \leq N-1 \} \\ &= O(\max \{N^{1/2}k^N, Nk^N/\log \log p \})
\end{align*}

\begin{proof}
 the set $\mathcal{F}_N$ contains $k^N$ polynomials. For each $f \in \mathcal{F}_N$, we can choose a sequence $\Phi$ of terms in $\mathcal{F}$ whose $\Phi^{(N)}= f$, obtaining $k^N$ sequences representing the elements of $\mathcal{F}_N$.
 For each sequence $\Phi$ chosen,  \begin{center} $M_{w, \Psi}(t,N) = O(\max \{N^{1/2} , N/\log \log p \})$ \end{center}
uniformly for any $\Phi$ by Theorem~\ref{th4.4}, or in other words, for each path in the $N$-tree $\mathcal{F}_N$. Since there are $k^N$ paths(polynomials, sequences) in the $n$-tree $\mathcal{F}_N$, this yields
\begin{align*}
 \# \{ u \in \overline{\mathbb{F}}_p| u= f(w), \tau(u) \leq t ,f \in &\mathcal{F}_n, n \leq N-1 \} \\ &= O(\max \{N^{1/2}k^N, Nk^N/\log \log p \})
\end{align*}
\end{proof}
\end{Cor}

\begin{Th}Let $\mathcal{F}=\{ \phi_1,..., \phi_k \} \in \mathbb{Z}[X]$ be a set of polynomials of respective degrees $d_i \geq 2$ that are not linear conjugate to monomials neither to Chebyshev polynomials (non-special) and $d = \max_i \deg_i$, and let $s(w)$ be the minimum number of sequences $\Psi_i$'s of terms in $\mathcal{F}$ such that $\mathcal{O}_{\mathcal{F}}(w)= \mathcal{O}_{\Psi_1}(w) \cup ... \cup \mathcal{O}_{\Psi_{s(w)}}(w)$.
  
  Then for any prime $p$, for all but $O(1)$ initial values $w \in \overline{\mathbb{F}}_p$, we have
  \begin{center}
  $d^{T(w)}\tau(w)^{s(w)} \geq (O(\log p))^{s(w)}$.
  \end{center}

\end{Th}
\begin{proof}
 Given $\Psi$ a sequence of functions from $\mathcal{F}$, $w \in \overline{\mathbb{F}}_p$,  we make for instance the notation $T_\Psi(w):= \# \mathcal{O}_\Psi(w)$.
 
 It can be seen that for some integers $m_\Psi, l_\Psi, n$
 
\begin{center}
 $T(w) \geq T_\Psi(w) \geq  m_\Psi > l_\Psi \geq 0$ and $n=\tau(w)$
\end{center} so that $\Psi^{(m_\Psi)}(w)=\Psi^{(l_\Psi)}(w)$ and $\Phi_n(w)=0$. Taking $Q_{m_\Psi, l_\Psi,n}$ as the resultant of the polynomials $\Psi^{(m_\Psi)}(X)-\Psi^{(l_\Psi)}(X)$ and $\Phi_n(X)$, this implies that 
\begin{center}
 $p | Q_{m_\Psi, l_\Psi,n}$.
\end{center} If $|Q_{m_\Psi, l_\Psi,n}| <p$, then $Q_{m_\Psi, l_\Psi,n}=0$ and thus the polynomials $\Psi^{(m_\Psi)}(X)-\Psi^{(l_\Psi)}(X)$ and $\Phi_n(X)$ have a common root in $\mathbb{C}$. By Lemma~\ref{lem4.3}, there are only $O(1)$ possible values of $n$ for such, and therefore only finitely many possibilities of $w \in \overline{\mathbb{F}}_p$ for such.

Now we point that $d^{m_\Psi}n \leq d^{T_\Psi(w)}\tau(w)$, and use Lemma~\ref{lem4.1} with $r=n, s=1$ to conclude that
\begin{center}
 $|Q_{m_\Psi, l_\Psi,n}|= \exp(O(d^{m_\Psi}n)) \leq \exp(O(d^{T_\Psi(w)}\tau(w)))$, 
\end{center} where $O$ does not depend on which $\Psi$ was initially taken.
Hence, for $c_0$ small enough not depending on $\Psi$, $d^{T_\Psi(w)}\tau(w) \leq c_0 \log p$ implies $|Q_{m_\Psi, l_\Psi,n}| < p$, and then $d^{T_\Psi(w)}\tau(w) \geq O(\log p)$ for all but $O(1)$ values $w \in \overline{\mathbb{F}}_p$, where these constants do not depend on $\Psi$.

Let $s(w)$ be the minimum number of sequences $\Psi_i$'s of terms in $\mathcal{F}$ such that $\mathcal{O}_{\mathcal{F}}(w)= \mathcal{O}_{\Psi_1}(w) \cup ... \cup \mathcal{O}_{\Psi_{s(w)}}(w)$. Thus
\begin{center}
$d^{T(w)}\tau(w)^{s(w)}\geq (d^{T_{\Psi_1}(w)}\tau(w))...(d^{T_{\Psi_{s(w)}}(w)}\tau(w)) \geq (O(\log p))^{s(w)}$
\end{center} for all but $O(1)$ values of $w \in \overline{\mathbb{F}}_p$.
\end{proof}

\section{A graph theory result}\label{sec5} 
Here we present a graph theory result of M\'{e}rai and Shparlinski~\cite{MS} that will be used in our next result.

Let $\mathcal{H}$ be a directed graph with possible multiple edges. Let $\mathcal{V}(\mathcal{H})$ be the set of vertices of $\mathcal{H}$. For $u,v \in \mathcal{V}(\mathcal{H})$, let $d(u,v)$ be the distance from $u$ to $v$, that is, the length of a shortest (directed) path from $u$ to $v$. Assume, that all the vertices have the out-degree $k\geq 1$, and the edges from all vertices are labeled by $\{1,...,k \}$.

For a word $\omega \in \{1,...,k \}^*$ over the alphabet $\{ 1,...,k\}$ and $ u \in \mathcal{V}(\mathcal{H})$, let $\omega(u) \in \mathcal{V}(\mathcal{H})$ be the end point of the walk started from $u$ and following the edges according to $\omega$.

Let us fix $u \in \mathcal{V}(\mathcal{H})$ and a subset $\mathcal{A} \subset \mathcal{V}(\mathcal{H})$. Then for words $\omega_1,..., \omega_l$ put
\begin{align*}
 L_N(u,\mathcal{A}; \omega_1,...,\omega_l)= \# \{v &\in \mathcal{V}(\mathcal{H}) : d(u,v) \leq N,\\ &d(u, \omega_i(v)) \leq N , \omega_i(v) \in \mathcal{A}, i=1,...,l  \}.
\end{align*}
To state the results, for $k,h \geq 1$, let $B(k,h)$ denote the size of the complete $k$-tree of depth $h-1$, that is
\begin{center} $  
B(k,h) = 
     \begin{cases}
       \ h
       \ &\quad\text{if} ~ k=1, \\
       \
       \
       \ \dfrac{k^h-1}{k-1}
       \ &\quad\text{otherwise} ~. \\

     \end{cases}
 $ \end{center}
 
 \begin{Lemma}\label{lem5.1}
  Let $u \in \mathcal{V}(\mathcal{H})$, and $h,l \geq 1$ be fixed. If $\mathcal{A} \subset \mathcal{V}(\mathcal{H})$ is a subset of vertices with
  \begin{align*}
   \#\{ v \in \mathcal{A} : d&(u,v) \leq N \}\\ & \geq \max \left\{ 3B(k,h), \dfrac{3l}{h}\# \{v \in \mathcal{V}(\mathcal{H}) : d(u,v) \leq N \} \right\},
  \end{align*}
then there exist words $\omega_1,...,\omega_l \in \{ 1,...,k \}^*$ of length at most $h$ such that
\begin{center}
 $L_N(u,\mathcal{A}; \omega_1,...,\omega_l) \gg  \dfrac{h}{B(k,h)^{l+1}} \# \{v \in \mathcal{V}(\mathcal{H}) : d(u,v) \leq N \}$,
\end{center} where the implied constant depend only on $l$.

 \end{Lemma}
 \section{Another semigroup result about multiplicative orders}\label{sec6}
\begin{Th}\label{th6.1}
 Let $\mathcal{F}=\{ \phi_1,..., \phi_k \} \in \mathbb{Z}[X]$ be a set of polynomials of respective degrees $d_i \geq 2$ that are not linear conjugate to monomials neither to Chebyshev polynomials (non-special), and $d = \max_i \deg_i$.
Let also $h,l \geq 1$ be  integers such that $h \geq 3l$ and
 $\# \{ v \in  \overline{\mathbb{F}}_p  |\tau(v) \leq t, v = f(u), f \in \mathcal{F}_n, n \leq N   \} \geq 3B(k,h)$.
 Then, for any fixed $\epsilon >0$,
  
  (i) For any prime $p$ and $t\leq (\log p) ^{1/2 - \epsilon}$, for all initial values $u \in \overline{\mathbb{F}_p}$, we have 
  \begin{align*}
   \# \{ v \in \overline{\mathbb{F}}_p| v= f(u), \tau(u) \leq t ,f \in &\mathcal{F}_n, n \leq N \}\\ &\ll_{l, \mathcal{F}} \max \left\{ \dfrac{B(k,h)^{l+1}}{ h} ,  \dfrac{ B(k,h)^{l+1}}{ \log \log p} \right\}
  \end{align*}
  (ii) Let $\Psi$ be a sequence of terms in $\mathcal{F}$. Then for any sufficiently large $P \geq 1$ and $t \leq P^{1/2 - \epsilon}$, for all but $o(P/\log P)$ primes $p \leq P$, for all initial values $u \in \overline{\mathbb{F}}_p$, we have 
  \begin{align*}
   \# \{ v \in \overline{\mathbb{F}}_p| v= f(u), \tau(u) \leq t ,f \in &\mathcal{F}_n, n \leq N \}  \\ &\ll_{l, \mathcal{F}} \max \left\{ \dfrac{B(k,h)^{l+1}}{ h} ,  \dfrac{ B(k,h)^{l+1}}{  \log p} \right\}
  \end{align*}
 
 \end{Th}
\begin{proof} We make the notation 
\begin{center}                                   
$\Gamma:= \{ x \in \overline{\mathbb{F}}^*_p : \tau(y) \leq t \}$
\end{center}

 We consider the directed graph with the elements of $\Gamma$ as vertices, and edges $(x,\phi_i(x))$ for $i=1,...,k$ and $x \in \Gamma$. With the notation of Section~\ref{sec5} and Lemma~\ref{lem5.1}, we let $\Gamma$ take the place of $\mathcal{H}$ and $\mathcal{A}$.
 By hypothesis, $l \leq h/3$ and $\#\{ v \in \Gamma, d(u,v) \leq N \} \geq 3B(k,h)$. From Lemma~\ref{lem5.1}, there exist words $\omega_1,...,\omega_l \in \{ 1,...,k\}^*$ of length at most $h$, and therefore degree at most $d^h$, such that
 \begin{equation}\label{eq6.1}
  L_N(u, \Gamma; \omega_1,...,\omega_l) \geq c_1  \dfrac{h}{B(k,h)^{l+1}} \# \{v \in \mathcal{V}(\Gamma) : d(u,v) \leq N \},
 \end{equation} with $c_1$ a positive constant depending only on $l$.
 
 If $v \in L_N(u, \Gamma; \omega_1,...,\omega_l)$, then $\{ v, \omega_i(v) \} \subset \Gamma$ for each $i$. This means that, for a given $i=1,...,l$,
 \begin{center}
  $\Phi_r(v)=\Phi_s(w_i(v))=0$
 \end{center} for some $r,s \in [1,t]^2$.  Denote by $R_{r,s,\omega_i}$ the resultant of the polynomials $\Phi_r(X)$ and $ \Phi_s(\omega_i(X))$.
By Lemma~\ref{lem4.2},  there are at most $c_2$  values of $v \in  \overline{\mathbb{F}}_p$
solutions of $R_{r,s,\omega_i}=0$, and therefore $c_2 \geq L_N(u, \Gamma; \omega_1,...,\omega_l)$. 
If \begin{equation}\label{eq6.2}\# \{v \in \mathcal{V}(\Gamma) : d(u,v) \leq N \} > c_2B(k,h)^{l+1}/c_1 h, \end{equation} there exists a $(r,s,\omega_i)$ such that $p | R_{r,s,\omega_i} \neq 0$.

In this  case , using Lemma~\ref{lem4.1} and Proposition~\ref{prop2.1} with $d = \max_i d_i$, we have
\begin{center}
 $\log |R_{r,s,\omega_i}|=O(rsd^h)=O(t^2d^h)$,
\end{center} and this does not depend on $p$.

Thus, $\log p = O(t^2d^h)$, and as $t \leq (\log p)^{1/2 - \epsilon}$, we derive $(\log p)^{2 \epsilon} \leq t^{-2}\log p = O(d^h)$, and therefore
 $h^{-1} \leq c_3 (\log \log p)^{-1}$. In this case, by (\ref{eq6.1}),
 \begin{align*}
  \# \{v \in \mathcal{V}(\Gamma) : d(u,v) \leq N \} \leq \dfrac{c_2 B(k,h)^{l+1}}{c_1 h} \leq \dfrac{c_2 c_3 B(k,h)^{l+1}}{c_1 \log \log p}.
 \end{align*}

 If
 \begin{center}
  $  \# \{v \in \mathcal{V}(\Gamma) : d(u,v) \leq N \} \geq \max \{ \dfrac{2 c_2B(k,h)^{l+1}}{c_1 h} , 2 \dfrac{c_2 c_3 B(k,h)^{l+1}}{c_1 \log \log p} \}$,
 \end{center} then (\ref{eq6.1}) and (\ref{eq6.2}) yield a contradiction, from where we proved part (i).

 For (ii), denoting by $\Omega(r)$ the number of distinct prime divisors of an integer $r \neq 0$, we see that if $R_{r,s, \omega_i} \neq 0$, then
\begin{center}
 $\Omega(R_{r,s,\omega_i}) \leq 2 \log |R_{r,s,\omega_i}| \leq O(t^2d^h) \leq O( P^{1- 2 \epsilon} d^h)$,
\end{center} and this does not depend on $p$.

We note that
$h^{-1} \geq \dfrac{\log d}{\epsilon \log P}$ implies that $\Omega(R_{r,s,\omega_i} = o(P/ \log P)$. Therefore, $\displaystyle\max_{\Phi} \rho_\Phi \leq O(1/\log p)$ for all but $o(P/\log P)$ primes in this case,
which implies the bound of part (ii).

\end{proof}

\begin{Rem}
If the hypothesis of Theorem~\ref{th6.1} are satisfied with $h= (\log_k N)^{\frac{1}{l+1}}$, then we recover and generalise Theorem 1.2 of \cite{S}.
\end{Rem}

\end{document}